\documentclass[10pt,a4paper]{article}
\usepackage[margin=1.1cm]{geometry}
\usepackage{amsmath,amsthm,amsfonts,amssymb,amscd,cite,graphicx}
\usepackage{url}
\usepackage{tikz}
\usepackage[table]{xcolor}
\usepackage{multicol}

\usepackage[OkabeIto]{colorblind}

\usepackage{titlesec}
\titleformat{\section}
{\normalfont\fontsize{12}{15}\bfseries}{\thesection}{1em.}{}

\usepackage[labelfont=bf]{caption}
\counterwithin{figure}{section}
\counterwithin{table}{section}

\newtheorem{prop}{Proposition}[section]
\newtheorem{conj}{Conjecture}[section]
\newtheorem{cor}{Corollary}[section]
\newtheorem{lem}{Lemma}[section]
\newtheorem{defi}{Definition}[section]

\newtheorem{rem}{Remark}[section]

\newtheorem{question}{Question}[section]

\allowdisplaybreaks[4]

\let\oldbibliography\thebibliography
\renewcommand{\thebibliography}[1]{%
  \oldbibliography{#1}%
  \setlength{\itemsep}{-2pt}%
}

\begin{document}

\baselineskip=0.20in

\makebox[\textwidth]{%
\hglue-15pt
\begin{minipage}{0.6cm}	
\vskip9pt
\end{minipage} \vspace{-\parskip}
\begin{minipage}[t]{6cm}
\footnotesize{ {\bf Discrete Mathematics Letters} \\ \underline{www.dmlett.com}}
\end{minipage}
\hfill
\begin{minipage}[t]{6.5cm}
\normalsize {\it Discrete Math. Lett.}  {\bf X} (202X) XX--XX
\end{minipage}}
\vskip36pt

\noindent
{\large \bf Generalizing OOOOOOB}\\

\noindent
Alon Danai$^{1,}$,
Paul Ellis$^{1,}\footnote{Email address: {\color{blue} paulellis@paulellis.org} }$, 
Thotsaporn Aek Thanatipanonda$^{2}$ \\

\noindent
\footnotesize $^1${\it Department of Mathematics, Rutgers University,
110 Frelinghuysen Road,
Piscataway, NJ,
08854,
USA}\\
\noindent
 $^2${\it Science Division, Mahidol University International College,
999 Phutthamonthon Sai 4 Rd, Salaya, Phutthamonthon District,
Nakhon Pathom,
73170,Thailand \/} \\

\noindent
 (\footnotesize Received: Day May 2026. Received in revised form: Day Month 202X. Accepted: Day Month 202X. Published online: Day Month 202X.)\\

\setcounter{page}{1} \thispagestyle{empty}

\baselineskip=0.20in

\normalsize

 \begin{abstract}
 \noindent
 We present three versions of the classic two-pile game \textsc{one-or-one-or-one-of-both} generalized to the multi-pile context.  In each case, we explore the resulting $\mathcal{P}$-positions.  In the first version, there is a simple pattern.  In the other two versions, we find partial solutions in each case through two experimental routes.  First by limiting the number of piles, then by limiting the number of tokens per pile.
 \\[2mm]
 \end{abstract}

\baselineskip=0.20in

\section{Introduction}
The game \textsc{one-or-one-or-one-of-both}, abbreviated \textsc{oooooob}, is a classic puzzle: 
\begin{quote}There are two piles of tokens.  On their turn, a player may take a token from one pile, take a token from the other pile, or take one token from each pile.  The last player to move wins.
\end{quote}
Students generally discover the following.
\begin{prop}
    In \textsc{oooooob}, the $\mathcal{P}$-positions are precisely those where both piles are even.
\end{prop}  

In this paper, we explore the $\mathcal{P}$-positions of three natural ways to generalize \textsc{oooooob} to more than two piles.  
\begin{itemize}
    \item In Version A, a player selects any non-empty set of piles and removes one token from each.
    \item In Version B, a player may remove any one token \emph{or} remove one token from each pile.
    \item In Version C, a player may remove any one token \emph{or} select two piles and remove one token from each.
\end{itemize}

Version A is easy to solve, and the proof is left as an exercise for the reader.

\begin{prop}
    In Version A, the $\mathcal{P}$-positions are those where every pile has an even number of tokens.
\end{prop}

Curiously, we can show that the all-even positions are $\mathcal{P}$-positions for every version, without having a full classification of $\mathcal{P}$-positions.

\begin{lem}\label{lemma - all evens is P}
    If all piles are even, then it is a $\mathcal{P}$-position in both Version B and Version C.
\end{lem}

\begin{proof}
We proceed by induction on the total number of tokens in the game, and prove both cases simultaneously.

First, note that the unique terminal position is a $\mathcal{P}$-position.  Next, let $k>0$, and consider a position $G$ with all even piles and $k$ total tokens. Let $H$ be an option of $G$.  In Version B, $H$ has either one odd pile, or all odd piles.  In Version C, $H$ has one or two odd piles.  In each case, there is an option $J$ of $H$ which has all even piles, so by induction, $J$ is a $\mathcal{P}$-position.  Thus $H$ is an $\mathcal{N}$-position.  Since the choice of $H$ was arbitrary, $G$ is a $\mathcal{P}$-position. 
\end{proof}

\begin{cor}\label{cor: one or two or all odds}
    If exactly one of the piles is odd, then it is an $\mathcal{N}$-position in both Version B and Version C.  In Version B, any position with all odd piles is an $\mathcal{N}$-position.  In Version C, any position exactly two odd piles is an $\mathcal{N}$-position.
\end{cor}

In the next two sections, we investigate the $\mathcal{P}$-positions for Versions B and C.  In each case, we obtain partial results in two directions.  First, we obtain results for $3$, $4$, or $5$ total piles.  Then we see what we can find if we instead limit the maximum pile size to $1$, $2$, or $3$ tokens.

This paper was inspired by recent work of Takayuki Morisawa \cite{Morisawa} on $m$-pile divisor \textsc{nim}.  As in traditional \textsc{nim}, the winning strategy for the multiple-pile version of that game is a generalization of the winning strategy for the two-pile version.  So we wondered, ``What about \textsc{oooooob}?''  

\textsc{oooooob} can be found, for example, in \cite[Chapter 7, Problem 23]{MathCircles} or \cite[Chapter 3, Section 2.3]{FiveFabulous}.  Note that two-pile \textsc{oooooob} is the two-dimensional vector subtraction game defined by $\{(-1,0),(0,-1),(-1,-1)\}$.  See \cite{Larsson} and \cite{SubtractionTransfer} for more on vector games generally.  

We used a computer to generate data for these results.  See the last author's website (\url{http://www.thotsaporn.com}) for the relevant Maple code.

\section{Version B}

For Version B, we start with two recursive results.  Let $G$, $H$, $J$, $U$, $A$, etc, denote elements of $\mathbb{N}^k$ for some $k\geq 0$.
We use \emph{unioption} and \emph{alloption} to denote options which are obtained by removing a single token and removing one token from each pile, respectively.  We use \emph{$\mathcal{P}$-option} (similarly, \emph{$\mathcal{P}$-unioption} and \emph{$\mathcal{P}$-alloption}) to denote an option which is itself a $\mathcal{P}$-position.  

\begin{lem}\label{lemma - Version B - add two 1s}
    If $G$ is nonempty, $1,1,G$ is of the same outcome class as $G$.
\end{lem}

Note that this is false if $G$ is empty, since $1,1$ is an $\mathcal{N}$-position.

\begin{proof}
    We proceed by induction on the total number of tokens in $G$.  For the base cases that $G$ has $1$ or $2$ total tokens, note that $1$; $1,1,1$; $1,1$; $1,1,1,1$ are all $\mathcal{N}$-positions, and that $2$ and $2,1,1$ are $\mathcal{P}$-positions.
    
    Suppose $G$ is a nonempty $\mathcal{P}$-position with at least $3$ total tokens.  Then the options of $1,1,G$ are 
    \begin{itemize}\begin{multicols}{3}
        \item $1,G$
        \item $H$
        \item $1,1,H$
        \end{multicols}
    \end{itemize}
     where $H$ is an option of $G$.  We will show that each of these is an $\mathcal{N}$-position.

    $1,G$ has the $\mathcal{P}$-option $G$, and $H$ is an option of $G$, so both of these are $\mathcal{N}$-positions.

    Suppose that $J$ is a $\mathcal{P}$-option of $H$.  If $J$ is an unioption, then by induction, $1,1,J$ is a $\mathcal{P}$-unioption of $1,1,H$.  If $J$ is the alloption of $H$, then it is also the alloption of $1,1,H$.  In either case, $1,1,H$ is an $\mathcal{N}$-position.

    Now suppose $G$ is a $\mathcal{N}$-position, and that $H$ is a $\mathcal{P}$-option of $G$.  If $H$ is a unioption, then by induction, $1,1,H$ is a $\mathcal{P}$-unioption of $1,1,G$.  If $H$ is the alloption of $G$, then it is also the alloption of $1,1,G$.  In either case, $1,1,G$ is an $\mathcal{N}$-position.    
\end{proof}

\begin{lem}\label{lemma - Version B - add two 2s}
    If $G$ has a pile of size of at least $2$, then  $2,2,G$ is of the same outcome class as $G$.
\end{lem}

\begin{proof}
    We proceed by induction on the total number of tokens in $G$.  For the base cases where $G$ has $2$ or $3$ total tokens, Lemma \ref{lemma - all evens is P} implies that $2$ and $2,2,2$ are both $\mathcal{P}$-positions.  Thus $2,1$ and $2,2,2,1$ are both $\mathcal{N}$-positions.
    
    Suppose $G$ is a $\mathcal{P}$-position with a pile of size at least $2$ and at least $4$ total tokens.  The options of $2,2,G$ are 
    \begin{itemize}\begin{multicols}{3}
        \item $1,2,G$
        \item $1,1,A$
        \item $2,2,U$
        \end{multicols}
    \end{itemize}
     where $A$ is the alloption and $U$ is a unioption of $G$.  We will show that all of these are $\mathcal{N}$-positions.

    By Lemma \ref{lemma - Version B - add two 1s}, $1,1,G$ is a $\mathcal{P}$-option of $1,2,G$, and $1,1,A$ is an $\mathcal{N}$-position.

    Let $U'$ be a $\mathcal{P}$-option of $U$.  Corollary \ref{cor: one or two or all odds} implies that $U'$ must have a pile of size $2$, so we may use induction to say that $2,2,U'$ is a $\mathcal{P}$-position. Then if $U'$ is a unioption of $U$, $2,2,U'$ is a $\mathcal{P}$-unioption of $2,2,U$.  On the other hand, if $U'$ is the alloption of $U$, then $1,1,U'$ is an option of $2,2,U$, and by Lemma \ref{lemma - Version B - add two 1s} is a $\mathcal{P}$-position.

    Now suppose $G$ is an $\mathcal{N}$-position, and $H$ is a $\mathcal{P}$-option of $G$.  If $H$ is the alloption, then $1,1,H$ is the alloption of $2,2,G$, and by Lemma \ref{lemma - Version B - add two 1s}, it is a $\mathcal{P}$-position.  If $H$ is a unioption, then $2,2,H$ is a unioption of $2,2,G$, and by induction is a $\mathcal{P}$-position.  In either case, $2,2,G$ is an $\mathcal{N}$-position.
\end{proof}

\subsection{Version B: Few piles}

\begin{defi}
    Let $e_k$ (or $o_k$) denote any even (or odd) pile of size at least $k$.
\end{defi}

\begin{lem}\label{lemma - Version B - 3 piles} The $\mathcal{P}$-positions for $3$ total piles are as follows.
\begin{itemize}
    \item $e_2,e_2,e_2$
    \item $e_2,o_1,o_1$
\end{itemize}
\end{lem}

\begin{proof}
    By Lemma \ref{lemma - all evens is P} and Corollary \ref{cor: one or two or all odds}, $e_2,e_2$ and $e_2,e_2,e_2$ are $\mathcal{P}$-positions, and $e_2,o_1$; $o_1,o_1$;  $o_1,o_1,o_1$ and $e_2,e_2,o_1$ are $\mathcal{N}$-positions.  Thus $e_2,o_1,o_1$ is a $\mathcal{P}$-position.  
\end{proof}

\begin{lem}\label{lemma - Version B - 4 piles} The $\mathcal{P}$-positions for $4$ total piles are as follows.
\begin{itemize}
    \item $1,1,e_2,e_2$
    \item $e_2,e_2,e_2,e_2$
    \item $e_2,o_3,o_3,o_3$
\end{itemize}
\end{lem}

\begin{rem}
    This lemma implies that there is no version of Lemmas $\ref{lemma - Version B - add two 1s}$ and $\ref{lemma - Version B - add two 2s}$ for $3,3,G$.  For example, $e_2,o_3$ is an $\mathcal{N}$-position, while $3, 3, e_2,o_3$ is a $\mathcal{P}$-position
\end{rem}

\begin{figure}[ht!]   \centering
    \begin{tikzpicture}[scale=.5]
    \node (A) at (0,0) {$\fbox{1 $e_2$ $e_2$ $e_2$}\;\mathcal{N}$};
    \node (B) at (5,0) {$\fbox{1 $o_3$ $o_3$ $o_3$}\;\mathcal{N}$};
    \node (C) at (10,0) {$\fbox{1 $e_2$ $o_3$ $o_3$}\;\mathcal{N}$};
    \node (D) at (15,0) {$\fbox{1 1 $e_2$ $e_2$}\;11\mathcal{P}$};
    \node (E) at (2,4) {$(\star)\;\fbox{$e_2$ $o_3$ $o_3$ $o_3$}$};
    \node (F) at (10,4) {$(\star)\;\fbox{$e_2$ $e_2$ $o_3$ $o_3$}$};
    \node (G) at (0,8) {$\fbox{$o_3$ $o_3$ $o_3$ $o_3$}\;\mathcal{N}$};
    \node (H) at (6,8) {$\fbox{$e_2$ $e_2$ $e_2$ $o_3$}\;\mathcal{N}$};
    \node (I) at (12,8) {$\fbox{$e_2$ $e_2$ $e_2$ $e_2$} \;\mathcal{P}$};
    \draw[->] (E) to (A);
    \draw[->] (E) to (B);
    \draw[->] (F) to (B);
    \draw[->] (F) to (C);
    \draw[->, dashed] (F) to (D);
    \draw[->] (E) to (G);
    \draw[->] (E) to (H);
    \draw[->] (F) to (H);
    \draw[->] (G) to [bend left] (I);
    \draw[->] (H) to [bend left] (I);
    \draw[->] (F) to [bend left] (E);
    \draw[->] (E) to [bend left] (F);
    \path (F) edge [loop right] node { } (F);
    \node (J) at (0,-4) {$\fbox{$e_2 e_2 e_2$}\;\mathcal{P}$};
    \node (K) at (5,-4) {$\fbox{$e_2$ $o_1$ $o_1$}\;\mathcal{P}$};
    \node (L) at (10,-4) {$\fbox{1 $e_2$ $e_2$ $o_3$}\;\mathcal{N}$};
    \draw[->] (A) to (J);
    \draw[->] (B) to (J);
    \draw[->] (C) to (K);
    \draw[->] (L) to [bend left] (K);
    \node at (-5,4) {$\fbox{1 1 $e_2$ $o_3$}\;11\mathcal{N}$};
    \node at (-5,2) {$\fbox{1 1 $o_3$ $o_3$}\;11\mathcal{N}$};
    \node at (-5,0) {$\fbox{1 1 1 $e_2$}\;11\mathcal{N}$};
    \node at (-5,-2) {$\fbox{1 1 1 $o_3$}\;11\mathcal{N}$};
    \node at (-5,-4) {$\fbox{1 1 1 1}\;11\mathcal{N}$};
    \end{tikzpicture}
    \caption{All positions of Version B with $4$ tokens, along with some relevant options.}
    \label{figure: Version B for 4 piles}
\end{figure}
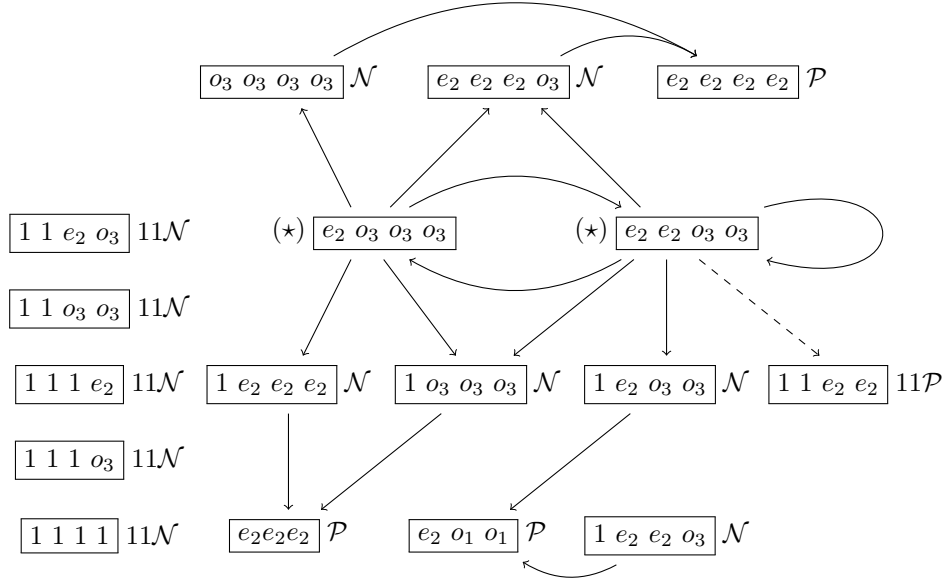

\begin{proof}  Refer to Figure \ref{figure: Version B for 4 piles}.  First note that all possible $4$-combinations of $1$, $e_2$, $o_3$ are included.  The positions marked $11\mathcal{P}$ and $11\mathcal{N}$ are $\mathcal{P}$- and $\mathcal{N}$-positions, respectively, by Lemma \ref{lemma - Version B - add two 1s}.  The positions marked as $\mathcal{P}$-positions are so by Lemma \ref{lemma - all evens is P} and/or \ref{lemma - Version B - 3 piles}.  The positions marked $\mathcal{N}$ each then have a $\mathcal{P}$-option as indicated.

The remaining positions are the two marked with $(\star)$, $e_2,o_3,o_3,o
_3$ and $e_2,e_2,o_3,o_3$.  We claim that $e_2,o_3,o_3,o_3$ is always a $\mathcal{P}$-position, and that $e_2,e_2,o_3,o_3$ is always an $\mathcal{N}$-position.  Note that all options of these positions are in the diagram.

We proceed by simultaneous induction on the total number of tokens.  The base case is then $2,2,o_3,o_3$, which has the $\mathcal{P}$-option $1,1,e_2, e_2$, shown as the dashed arrow in the diagram.  Then, in the inductive step, $e_2,o_3,o_3,o_3$ only has $\mathcal{N}$-options, and $e_2,e_2,o_3,o_3$  has the $\mathcal{P}$-option $e_2,o_3,o_3,o_3$.
\end{proof}

\begin{lem}
The $\mathcal{P}$-positions for $5$ total piles are as follows.
\begin{itemize}
    \begin{multicols}{2}
    \item $1,1,e_2,e_2,e_2$
    \item $1,1,e_2,o_1,o_1$
    \item $1,2,e_2,e_2,o_3$
    \item $1,e_4,e_4,o_5,o_5$
    \item $e_2,e_2,e_2,e_2,e_2$
    \item $e_2,o_3,o_3,o_3,o_3$
    \item $e_2,e_2,e_2,o_3,o_3$, with an exception:  $2,e_4,e_4,o_5,o_5$.
    \end{multicols}
\end{itemize}
\end{lem}

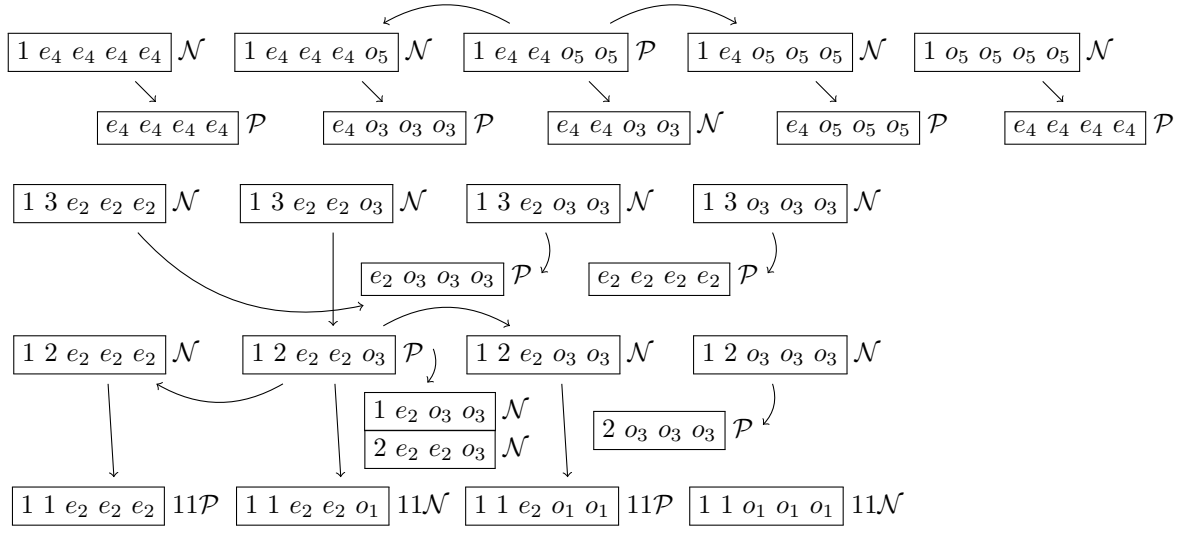
\begin{figure}[ht!]   \centering
    \begin{tikzpicture}[scale=.5]
    \node (A) at (0.25,0) {$\fbox{$1$ $1$  $e_2$ $e_2$ $e_2$}\;11\mathcal{P}$};
    \node (B) at (6.25,0) {$\fbox{$1$ $1$  $e_2$ $e_2$ $o_1$}\;11\mathcal{N}$};
    \node (C) at (12.25,0) {$\fbox{$1$ $1$  $e_2$ $o_1$ $o_1$}\;11\mathcal{P}$};
    \node (D) at (18.25,0) {$\fbox{$1$ $1$  $o_1$ $o_1$ $o_1$}\;11\mathcal{N}$};
    \node (F1) at (9,1.5) {$\fbox{$2$  $e_2$ $e_2$ $o_3$}\;\mathcal{N}$};
    \node (F2) at (9,2.5) {$\fbox{$1$  $e_2$ $o_3$ $o_3$}\;\mathcal{N}$};\node (H2) at (15,2) {$\fbox{$2$  $o_3$ $o_3$ $o_3$}\;\mathcal{P}$};
    \node (E) at (0,4) {$\fbox{$1$ $2$  $e_2$ $e_2$ $e_2$}\;\mathcal{N}$};
    \node (F) at (6,4) {$\fbox{$1$ $2$  $e_2$ $e_2$ $o_3$}\;\mathcal{P}$};
    \node (G) at (12,4) {$\fbox{$1$ $2$  $e_2$ $o_3$ $o_3$}\;\mathcal{N}$};
    \node (H) at (18,4) {$\fbox{$1$ $2$  $o_3$ $o_3$ $o_3$}\;\mathcal{N}$}; 
      \node (K1) at (9,6) {$\fbox{$e_2$ $o_3$ $o_3$ $o_3$}\;\mathcal{P}$};
      \node (L1) at (15,6) {$\fbox{$e_2$ $e_2$ $e_2$ $e_2$}\;\mathcal{P}$};
     \node (I) at (0,8) {$\fbox{$1$ $3$  $e_2$ $e_2$ $e_2$}\;\mathcal{N}$};
    \node (J) at (6,8) {$\fbox{$1$ $3$  $e_2$ $e_2$ $o_3$}\;\mathcal{N}$};
    \node (K) at (12,8) {$\fbox{$1$ $3$  $e_2$ $o_3$ $o_3$}\;\mathcal{N}$};
    \node (L) at (18,8) {$\fbox{$1$ $3$  $o_3$ $o_3$ $o_3$}\;\mathcal{N}$};
      \node (4A1) at (2,10) {$\fbox{$e_4$  $e_4$ $e_4$ $e_4$}\;\mathcal{P}$};
    \node (4B1) at (8,10) {$\fbox{$e_4$  $o_3$ $o_3$ $o_3$}\;\mathcal{P}$};
    \node (4C1) at (14,10) {$\fbox{$e_4$  $e_4$ $o_3$ $o_3$}\;\mathcal{N}$};
    \node (4D1) at (20,10) {$\fbox{$e_4$  $o_5$ $o_5$ $o_5$}\;\mathcal{P}$};
    \node (4E1) at (26,10) {$\fbox{$e_4$  $e_4$ $e_4$ $e_4$}\;\mathcal{P}$};
     \node (4A) at (0,12) {$\fbox{$1$ $e_4$  $e_4$ $e_4$ $e_4$}\;\mathcal{N}$};
    \node (4B) at (6,12) {$\fbox{$1$ $e_4$  $e_4$ $e_4$ $o_5$}\;\mathcal{N}$};
    \node (4C) at (12,12) {$\fbox{$1$ $e_4$  $e_4$ $o_5$ $o_5$}\;\mathcal{P}$};
    \node (4D) at (18,12) {$\fbox{$1$ $e_4$  $o_5$ $o_5$ $o_5$}\;\mathcal{N}$};
    \node (4E) at (24,12) {$\fbox{$1$ $o_5$ $o_5$ $o_5$ $o_5$}\;\mathcal{N}$};
    \draw[->] (E) to (A);
    \draw[->] (F) to (B);
    \draw[->] (F) to [bend left] (E);
    \draw[->] (F) to [bend left] (G);
    \draw[->] (F) to [bend left] (F2);
    \draw[->] (G) to (C);
    \draw[->] (H) to [bend left] (H2);
    \draw[->] (J) to (F);
    \draw[->] (I) to [bend right] (K1);
    \draw[->] (K) to [bend left] (K1);
    \draw[->] (L) to [bend left] (L1);
    \draw[->] (4A) to (4A1);
    \draw[->] (4B) to (4B1);
    \draw[->] (4C) to (4C1);
    \draw[->] (4D) to (4D1);
    \draw[->] (4E) to (4E1);
    \draw[->] (4C) to [bend right] (4B);
    \draw[->] (4C) to [bend left] (4D);
    \end{tikzpicture}
    \caption{All positions of Version B with $5$ tokens and at least one pile of size $1$, along with some relevant options.}
    \label{figure: Version B for 5 piles, part 1}
\end{figure}

\begin{proof}
    Our base case is any position with $5$ piles where at least one of the piles is $1$.  Note that these are organized by the size of the second smallest pile in Figure \ref{figure: Version B for 5 piles, part 1}, in the four rows which start on the far left.      
    The outcome classes of the positions in the bottom row are determined by Lemma \ref{lemma - Version B - add two 1s} and Lemma \ref{lemma - Version B - 3 piles}.  After that, for each $\mathcal{N}$-position, a $\mathcal{P}$-option is displayed, and each option of each $\mathcal{P}$-position is displayed.  The relevant options which have four piles are justified by the previous lemma.

    Assuming now that all piles have size at least $2$, we organize the remaining cases as follows.  The last four cases subdivide the case with two odds and three evens: one $2$ and no other $2$s or $3$s, one $2$ and another $2$, one $2$ and a $3$, no $2$s.  As a reminder, we mark the claimed outcome class of each.  We will analyze the ones marked A first, then the rest together with a simultaneous induction.
    \begin{itemize}
        \begin{multicols}{2}
        \item $e_2, e_2, e_2, e_2, e_2$ -- $\mathcal{P}$ -- A --  by Lemma \ref{lemma - all evens is P}
        \item $e_2, e_2, e_2, e_2, o_3$ -- $\mathcal{N}$ -- A --  by Corollary \ref{cor: one or two or all odds}
        \item $e_2, e_2,o_3,o_3,o_3$ -- $\mathcal{N}$
        \item $e_2,o_3,o_3,o_3,o_3$ -- $\mathcal{P}$ 
        \item $o_3, o_3, o_3, o_3, o_3$ -- $\mathcal{N}$ -- A --by Corollary \ref{cor: one or two or all odds}
        \item $2,e_4,e_4,o_5,o_5$ -- $\mathcal{N}$ -- A
        \item $2,2, e_2,o_3, o_3$ -- $\mathcal{P}$ -- A
        \item $2,3,e_2,e_2,o_3$  -- $\mathcal{P}$
        \item $e_4,e_4,e_4,o_5,o_5$ -- $\mathcal{P}$
        \end{multicols}
\end{itemize}
    Observe that $2,e_4,e_4,o_5,o_5$ has the $\mathcal{P}$-option $1,e_4,e_4,o_5,o_5$, so it is an $\mathcal{N}$-position.
    
    The only options of $2,2, e_2,o_3, o_3$ are the previously established $\mathcal{N}$-positions $1,1,e_2,e_2,o_1$; $1,2,e_2,o_3, o_3$;  $2,2,e_2,e_2,o_1$ (Corollary \ref{cor: one or two or all odds}); or the position $2,2,o_1,o_3,o_3$, which is given by Lemmas \ref{lemma - Version B - add two 2s} and \ref{lemma - Version B - 3 piles}.

    We claim that the four remaining cases are $\mathcal{P}$-positions except $e_2, e_2,o_3,o_3,o_3$, proceeding by induction on the total number of tokens.
    The base case is then  $2,2,2,3,3$, which is a $\mathcal{P}$-position by Lemmas \ref{lemma - Version B - add two 2s} and \ref{lemma - Version B - 3 piles}.

    Next if $e_2, e_2,o_3,o_3,o_3 = 2,2,o_3,o_3,o_3$, then it has the $\mathcal{P}$-option $1,1,e_2,e_2,e_2$.  Otherwise it has the option $2,o_3,o_3,o_3,o_3$, which, by induction, is a $\mathcal{P}$-position.

    We show the remaining cases are $\mathcal{P}$-positions by listing their options (with the alloption last) and noting that they are all $\mathcal{N}$-options.  
    For $e_2,o_3,o_3,o_3,o_3$, we have $o_1,o_3,o_3,o_3,o_3$; $e_2, e_2,o_3,o_3,o_3$; and $o_1,e_2,e_2,e_2,e_2$.
    For $2,3,e_2,e_2,o_3$, we have $1,3,e_2,e_2,o_3$; $2,2,e_2,e_2,o_3$; $2,3,e_2,e_2,e_2$; 
    $1,2,3,e_2,o_3$; 
     $e_2, e_2,o_3,o_3,o_3$; and $1, 2, o_1, o_1, e_2$.
    For $e_4,e_4,e_4,o_5,o_5$, we have $e_4,e_4,e_4,e_4,o_5$ and $e_2, e_2,o_3,o_3,o_3$.
\end{proof}

The situation for $6$ piles is more tedious to verify.  Once enough base cases are verified, the general case becomes clear.  Hence we leave the proof to the reader.

\begin{lem}
The $\mathcal{P}$-positions for $6$ total piles are as follows.  
\begin{itemize}
    \begin{multicols}{2}
    \item $1, 1, 1, 1, e_2, e_2$
    \item $1, 1,  e_2,  e_2, e_2, e_2$
    \item $1, 1,  e_2,  o_3, o_3, o_3$
    \item $1, 2,  e_4, e_4, e_4, o_3$
    \item $1, 3, e_6, e_6, o_5, o_5$
    \item $2, 3, e_6, e_6, e_6, o_5$
    \item $2, e_2, e_2, o_3, o_3, o_3$, except $2, 3, e_6, e_6, o_5, o_5$
    \item $e_2, e_2, e_2, e_2, e_2, e_2$
    \item $e_4, e_4, e_4, e_4, o_3, o_3$
    \item $e_2, o_3, o_3, o_3, o_3, o_3$
    \end{multicols}
\end{itemize}
\end{lem}
\begin{proof}[proof sketch]
Suppose sufficiently many base cases have been verified, perhaps by computer.  We proceed by induction on the total number of tokens.  Suppose we have a position with no ones, twos, or threes, and that the result has been verified for all positions with fewer total tokens.  We already know that positions with no odd piles are $\mathcal{P}$-positions. If there are $1$ or $6$ odd piles, there is an option with no odd piles, so it is an $\mathcal{N}$-position.  If there are $3$ or $4$ odd piles, there is an option with $2$ odd piles, so it is an $\mathcal{N}$-position.  If there are $2$ or $5$ odd piles, then there are only options with $1$, $3$, $4$, or $6$ odd piles, so it is a $\mathcal{P}$-position.
\end{proof}

Based on our findings for $3$, $4$, $5$, and $6$ piles, we conjecture the following. We have verified this by computer up to $11$ piles, up to $10$ tokens per pile.

\begin{conj}
    Suppose we have $k\geq 3$ piles of tokens.  Let $m=\lceil\frac{k}{3}\rceil$. Suppose each pile has more than $m$ tokens.  Then it is a $\mathcal{P}$-position precisely when:
\begin{itemize}
    \item If $k$ odd, then the number of odd piles is even
    \item If $k=0 \bmod 4$, then the number of odd piles is $0$, $2$, $4$,  \ldots $\frac{k}{2}-2$,   $\frac{k}{2}+1$ \ldots, $k-5$, $k-3$, $k-1$.
    \item If $k=2 \bmod 4$, then the number of odd piles is $0$, $2$, $4$,  \ldots $\frac{k}{2}-1$,   $\frac{k}{2}+2$ \ldots, $k-5$, $k-3$, $k-1$.
\end{itemize}
\end{conj}

\subsection{Version B: Few tokens per pile}

\begin{lem}
Suppose that all piles have at most $3$ tokens.  In particular, there are $a_1$ piles of size $1$, $a_2$ piles of size $2$, and $a_3$ piles of size $3$.  Then the nonempty $\mathcal{P}$ positions are $(a_1, a_2,a_3)=$
\begin{itemize}
    \item $(e_0, \geq 1, 0)$
    \item $(o_1, o_1, 1)$
    \item $(e_0, o_1, 
    \geq 2)$
\end{itemize}
\end{lem}

\begin{proof}
    First assume $a_3=0$. If $a_2=0$ and $a_1>0$, then Corollary \ref{cor: one or two or all odds} says we have an $\mathcal{N}$-position.
    So fix $a_2>0$ and proceed by induction on the total number of tokens.  If $a_1=0$, Lemma \ref{lemma - all evens is P} shows that we have a $\mathcal{P}$-position.  If $a_1>0$ is odd, then by induction, the option to remove a pile of size $1$ is a $\mathcal{P}$-position.
    If $a_1>0$ is even, then by induction, all unioptions are $\mathcal{N}$-positions, while Corollary \ref{cor: one or two or all odds} shows that the alloption is, as well.

    By this first case, if $a_3>0$, the alloption is a $\mathcal{P}$-position precisely when $a_2$ is even, so in the remaining cases, if $a_2$ even, we have an $\mathcal{N}$-position.
    
    Next assume $a_3=1$ and that $a_2$ is odd. If $a_1$ is even, taking one token from the pile of size $3$ is a $\mathcal{P}$-option.  If $a_1$ is odd, then the possible unioptions are $(a_1,a_2,a_3)=(o_1,e_2,0), (e_0,e_1,1), (e_0,o_1,1)$, each of which is a $\mathcal{P}$-position.

    Finally assume $a_3\geq 2$ and $a_2$ is odd. 
    If $a_1$ is odd, then by induction, the option to take one of the piles of size $1$ is a $\mathcal{P}$-position. 
    If $a_1$ is even, then the possible unioptions are $(a_1,a_2,a_3)=(e_0,e_2,1), (e_0,e_2,\geq 2), (o_1,e_2,\geq 2), (o_1,o_1,\geq 2)$, each of which is a $\mathcal{P}$-position. 
\end{proof}

We offer the next two cases without proof.  The cases were discovered by computer, and the verification, while long, is routine, and would violate the page limit of this journal.

\begin{lem}
Suppose that all piles have at most $4$ tokens, and that there is a pile of size $4$.  In particular, there are $a_i$ piles of size $i$.  Then the nonempty $\mathcal{P}$ positions are as follows.  $(a_1, a_2, a_3, a_4) =$
\begin{itemize}
    \item Base cases: $(o, e, 1, 1)$, $(e, e, \geq 2, 1)$,  $(e, e, 0, \{1,2\})$,  $(o, \geq 1, e_4, 2)$,  $(e, o, o_3, 2)$,  $(e, o, 2, 2)$, $(e,o,0,\{1,2,3\})$,  $(o, o, 1, \{2,3\})$

    \item  General Cases: $(e, e, e, \geq 3)$,   $(e, o, e, \geq 4)$,  $(o, o, o, \geq 4)$
    \end{itemize}
\end{lem}

\begin{lem}
Suppose that all piles have at most $5$ tokens, and that there is a pile of size $5$.  In particular, there are $a_i$ piles of size $i$.  Then the nonempty $\mathcal{P}$ positions are as follows.  $(a_1, a_2, a_3, a_4, a_5) =$
\begin{itemize}
    \item 23 base cases:
    $(e, e, e, e_4, 2)$, $(e, e, e, o, 2)$,$(e, e, e, 1, \{1, 3\})-(e, e, 0, 1, 1)$, $(e, e, o, e_4, \{1, 3\})$,$(e, e, o, 1, \{2, 4\})$, 
    \newline
    $(e, e, o, o, \{1, 3\}) - (e, e, o, 3, 3)$, $(e, o, e, e_4, 4)$, $(e, o, e, 0, \{2, 4\})$, $(e, o, e, \{0, 2\}, o) - (e, o, 0, \{0, 2\}, 1)$, $(e, o, e, o_5, \{2, 4\})$, $(e, o, e, 3, 3)$, $(e, o, o, \{0, 2\}, \{2, 4\})$, $(e, o, o, e_6, 3)$, $(e, o, o, e_4, 1)$, $(e, o, o, 0, \{1, 3\})$,
$(e, o, 1, 2, 1)$, $(e, o, o, o_5, \{1, 3\})$, 
\newline
$(o, e, e, 2, 2)$, $(o, e, e, o_5, 3)$, $(o, e, 0, 1, 1)$, $(o, e, o, e_4, 2)$, $(o, e, o, 2, \{1, 3\}) - (o, e, 1, 2, 1)$, $(o, e, o, o_3, 4)$, 
\newline
$(o, o, e, 2, \{2, 4\}) - (o, o, 0, 2, 2)$, $(o, o, 0, {0, 2}, 1)$, $(o, o, e, e_4, \{1, 3, 5\}) - (o, o, e, 4, 3)$, $(o, o, e, o_5, \{1, 3\})$, $(o, o, 0, 3, 1)$,
\newline
$(o, o, o, e_4, 4)$, $(o, o, o, 2, \{1, 3\}) - (o, o, 1, 2, 1)$, $(o, o, o, o_5, 2)$, $(o, o, o, 3, 3)$
    \item General Cases: $(e,e,e,o,\geq 4)$, $(e,e,o,o,\geq 5)$, $(e,o,o,e,\geq 5)$, $(e,o,e,e,\geq 6)$
    \end{itemize}
\end{lem}

\noindent For piles of size up to $6$, we have preliminary calculations which show that the general case for $\mathcal{P}$-positions  when $a_6\geq 8$ is \begin{itemize}
\item $a_5+a_3$ is even and $a_1$ is even; or
\item $a_5+a_3$, $a_2$, and $a_1$ are all odd
\end{itemize}

All of the results of this section follow a pattern, and we conjecture accordingly.

\begin{conj}
    Let $n\geq 3$. Suppose we have a position with $a_i$ piles of size $i$ for $1\leq i\leq n$, and that $a_n\geq 2n-4$.  Then the outcome class is a function of the parities of $a_1,a_2,\ldots a_{n-1}$.
\end{conj}

\section{Version C}
\subsection{Version C: Few piles}

\begin{defi}
    Suppose $x_i\in\{e,o\}$ for all $1\leq i\leq n$.  Let $\langle x_1, x_2, \ldots , x_n\rangle$ denote the set of all positions $a_1,a_2, \ldots, a_n$ where $a_1\leq a_2\leq \ldots \leq a_n$ and $x_i$ is the parity of $a_i$ for all $1\leq i\leq n$.
\end{defi}

\begin{lem}  If there are $3$, $4$, or $5$ piles, the $\mathcal{P}$-positions are as follows, where the piles are listed in nondecreasing order.
%
\begin{itemize}
    \begin{multicols}{2}
    \item $\langle e,e,e\rangle$
    \item $\langle o,o,o\rangle$
%
    \item $\langle e,e,e, e\rangle$
    \item $\langle e, o, o, o\rangle$
%
    \item $\langle e,e,e, e, e\rangle$
    \item $\langle e, e, o, o, o\rangle$
    \item $\langle o, o,e, e, o\rangle$
    \item $\langle o, o, o, o, e\rangle$
    \end{multicols}
\end{itemize}
\end{lem}


\begin{proof}
    For each of $3$, $4$, and $5$ piles, recall that Lemma \ref{lemma - all evens is P} and Corollary \ref{cor: one or two or all odds} say that each all-evens position is a $\mathcal{P}$-position, and each position with exactly one or two odd piles is an $\mathcal{N}$-position.
    
    In the case of $3$ piles, the only remaining position, $\langle e,e,e\rangle$, is then an $\mathcal{N}$-position.

    Next assume there are $4$ piles. We prove the remaining cases by induction on the number of tokens.  As a base case, note that $\langle 1,1,1,1\rangle$ is an $\mathcal{N}$-position.  If there are no odd piles, then taking one from the smallest pile is an $\mathcal{N}$-option. In the cases of $\langle o,e,o,o\rangle$, $\langle o,o,e,o\rangle$, $\langle o,o,o,e\rangle$, taking  one from the even pile and one from the smallest odd pile is an $\mathcal{N}$-option.  Finally, no options of $\langle e,o,o,o\rangle$ are $\mathcal{N}$-positions.

    Next assume there are $5$ piles.  We proceed by induction, adding the $\mathcal{P}$-position  $\langle 0,0,1,1,1\rangle$ to our collection of base cases.  Next, the following are all the remaining claimed $\mathcal{N}$-positions. In each case, a winning move is indicated by taking a token from the pile(s) which are capitalized.
    \[\langle e,O,o,o,o\rangle, 
    \langle O,e,o,o,o\rangle, 
    \langle o,o,e,O,o\rangle, 
    \langle o,o,o,E,O\rangle, 
    \langle e,O,E,o,o\rangle, 
    \langle e,O,o,E,o\rangle, 
    \langle E,o,o,o,e\rangle
    \] 
    \[
    \langle o,E,e,O,o\rangle, 
    \langle O,e,o,E,o\rangle, 
    \langle O,e,o,o,E\rangle, 
    \langle o,o,E,o,e\rangle,
    \langle o,o,o,E,e\rangle,
    \langle O,O,o,o,o\rangle
    \]
    It is then routine to check that there is no move to from one of the claimed $\mathcal{P}$-positions to another one.
\end{proof}

 Unlike the cases for $5$ or fewer piles, the case of $6$ piles does not follow a strict parity pattern. The data in this case is more chaotic, and as a demonstration, we report the $\mathcal{P}$-positions in the case that there are at least two piles of size $1$.  Again, we omit the tedious and routine proof.

\begin{lem} Suppose there are six piles, and at least two of the piles has size $1$.  Then the $\mathcal{P}$-positions are as follows.  \begin{itemize}
    \begin{multicols}{2}
        \item $\langle 1,1,1,1,1,1\rangle$
        \item $\langle 1,1,1,1,e,o\rangle$
        \item $\langle 1,1,1,x,y,y\rangle$, where $x$ and $y$ have same parity
        \item $\langle 1,1,e,e,o,o\rangle$, last two not equal
        \item $\langle 1,1,x,x,y,y+1\rangle$ with $x$ odd, $y$ even
        \item $\langle 1,1,e,o,e,e\rangle$ with last two not equal
        \item $\langle 1,1,x,y,z,z\rangle$ with $x$ odd, $x\neq y$, $y$ and $z$ same parity
    \end{multicols}
    \end{itemize}
\end{lem}

 \subsection{Version C: Few tokens per pile}

\begin{lem}\label{lemma - Version C for small piles}
Suppose that all piles have at most $3$ tokens.  In particular, there are $a_1$ piles of size $1$, $a_2$ piles of size $2$, and $a_3$ piles of size $3$.  Then the $\mathcal{P}$ positions are as follows:

\begin{itemize}
    \item $a_1=0$,  $a_3=0$ $(\colorbox{OI1}{$\mathcal{P}$})$
    \item $(a_1,a_2,a_3)=(0,1,3),(0,2,3),(1,0,2),(1,0,5),(2,0,1)$ $(\colorbox{OI4}{$\mathcal{P}$})$
    \item $a_1+2\cdot a_2 =0\bmod 3$, with the following exceptions:
    \begin{itemize}
        \item[\colorbox{OI2}{$\star$}] $(a_1,a_2,a_3)=(0,3k,e)$, where $e=1,2$
        \item[\colorbox{OI2}{$\star$}] $(a_1,a_2,a_3)=(1,3k+1,e)$, where $e=0,1$
        \item[\colorbox{OI2}{$\star$}] $(a_1,a_2,a_3)=(2,3k+2,e)$, where $e=0$
        \item[\colorbox{OI3}{$\dagger$}] $(a_1,a_2,a_3)=(0,0,4),(0,0,5),(1,1,2),(1,1,3),(1,1,4),(1,1,5)
,(2,2,3),(3,0,1),(3,0,2),(3,0,5)$
    \end{itemize}
\end{itemize}
\end{lem}

\begin{figure}[ht!]
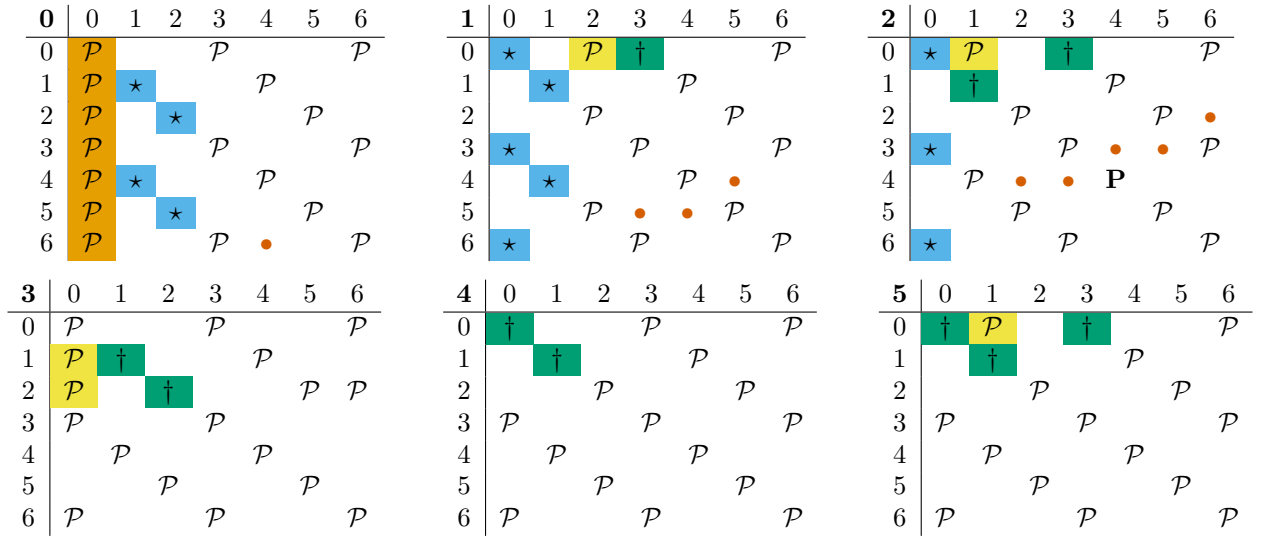
   \centering
$\begin{array}{c|ccccccc}
\mathbf{0}&0&1&2&3&4&5&6\\
\hline
0&\cellcolor{OI1}\mathcal{P}&&&\mathcal{P}&&&\mathcal{P}\\
1&\cellcolor{OI1}\mathcal{P}&\cellcolor{OI2}\star&&&\mathcal{P}&&\\
2&\cellcolor{OI1}\mathcal{P}&&\cellcolor{OI2}\star&&&\mathcal{P}&\\
3&\cellcolor{OI1}\mathcal{P}&&&\mathcal{P}&&&\mathcal{P}\\
4&\cellcolor{OI1}\mathcal{P}&\cellcolor{OI2}\star&&&\mathcal{P}&&\\
5&\cellcolor{OI1}\mathcal{P}&&\cellcolor{OI2}\star&&&\mathcal{P}&\\
6&\cellcolor{OI1}\mathcal{P}&&&\mathcal{P}&\textcolor{OI6}{\bullet}&&\mathcal{P}
\end{array}$
\qquad
$\begin{array}{c|ccccccc}
\mathbf{1}&0&1&2&3&4&5&6\\
\hline
0&\cellcolor{OI2}\star&&\cellcolor{OI4}\mathcal{P}&\cellcolor{OI3}\dagger&&&\mathcal{P}\\
1&&\cellcolor{OI2}\star&&&\mathcal{P}&\\
2&&&\mathcal{P}&&&\mathcal{P}\\
3&\cellcolor{OI2}\star&&&\mathcal{P}&&&\mathcal{P}\\
4&&\cellcolor{OI2}\star&&&\mathcal{P}&\textcolor{OI6}{\bullet}\\
5&&&\mathcal{P}&\textcolor{OI6}{\bullet}&\textcolor{OI6}{\bullet}&\mathcal{P}\\
6&\cellcolor{OI2}\star&&&\mathcal{P}&&&\mathcal{P}
\end{array}$
\qquad
$\begin{array}{c|ccccccc}
\mathbf{2}&0&1&2&3&4&5&6\\
\hline
0&\cellcolor{OI2}\star&\cellcolor{OI4}\mathcal{P}&&\cellcolor{OI3}\dagger&&&\mathcal{P}\\
1&&\cellcolor{OI3}\dagger&&&\mathcal{P}&\\
2&&&\mathcal{P}&&&\mathcal{P}&\textcolor{OI6}{\bullet}\\
3&\cellcolor{OI2}\star&&&\mathcal{P}&\textcolor{OI6}{\bullet}&\textcolor{OI6}{\bullet}&\mathcal{P}\\
4&&\mathcal{P}&\textcolor{OI6}{\bullet}&\textcolor{OI6}{\bullet}&\mathbf{P}&\\
5&&&\mathcal{P}&&&\mathcal{P}\\
6&\cellcolor{OI2}\star&&&\mathcal{P}&&&\mathcal{P}
\end{array}$

\medskip

$\begin{array}{c|ccccccc}
\mathbf{3}&0&1&2&3&4&5&6\\
\hline
0&\mathcal{P}&&&\mathcal{P}&&&\mathcal{P}\\
1&\cellcolor{OI4}\mathcal{P}&\cellcolor{OI3}\dagger&&&\mathcal{P}&\\
2&\cellcolor{OI4}\mathcal{P}&&\cellcolor{OI3}\dagger&&&\mathcal{P}&\mathcal{P}\\
3&\mathcal{P}&&&\mathcal{P}&&\\
4&&\mathcal{P}&&&\mathcal{P}&\\
5&&&\mathcal{P}&&&\mathcal{P}\\
6&\mathcal{P}&&&\mathcal{P}&&&\mathcal{P}
\end{array}$
\qquad
$\begin{array}{c|ccccccc}
\mathbf{4}&0&1&2&3&4&5&6\\
\hline
0&\cellcolor{OI3}\dagger&&&\mathcal{P}&&&\mathcal{P}\\
1&&\cellcolor{OI3}\dagger&&&\mathcal{P}&\\
2&&&\mathcal{P}&&&\mathcal{P}\\
3&\mathcal{P}&&&\mathcal{P}&&&\mathcal{P}\\
4&&\mathcal{P}&&&\mathcal{P}&\\
5&&&\mathcal{P}&&&\mathcal{P}\\
6&\mathcal{P}&&&\mathcal{P}&&&\mathcal{P}
\end{array}$
\qquad
$\begin{array}{c|ccccccc}
\mathbf{5}&0&1&2&3&4&5&6\\
\hline
0&\cellcolor{OI3}\dagger&\cellcolor{OI4}\mathcal{P}&&\cellcolor{OI3}\dagger&&&\mathcal{P}\\
1&&\cellcolor{OI3}\dagger&&&\mathcal{P}&\\
2&&&\mathcal{P}&&&\mathcal{P}\\
3&\mathcal{P}&&&\mathcal{P}&&&\mathcal{P}\\
4&&\mathcal{P}&&&\mathcal{P}&\\
5&&&\mathcal{P}&&&\mathcal{P}\\
6&\mathcal{P}&&&\mathcal{P}&&&\mathcal{P}
\end{array}$


    \caption{Positions of Version C with small piles.  In each grid, the number in the top-left corner is $a_3$, the row is $a_2$, and the column is $a_1$. Exceptions to the general case are highlighted.  The position $(4,4,2)$ is marked with $\mathbf{P}$ and its options are marked with $\textcolor{OI6}{\bullet}$.
    }
    \label{figure: Version C for small piles}
\end{figure}

\begin{proof}
    Figure \ref{figure: Version C for small piles} shows all $\mathcal{P}$-positions for $a_1,a_2\leq 6$, $a_3\leq 5$.  As a visual guide, the options of $(4,4,2)$ are marked with $\textcolor{OI6}{\bullet}$.  
    It is then easily checked that the general case holds also for $a_1,a_2\leq 6$ and $a_3=6,7$. These are enough base cases for the following induction.

    By Lemma \ref{lemma - all evens is P} $a_1,a_3=0$ gives a $\mathcal{P}$-position.  Thus any position with $1\leq a_1+a
    _3\leq 2$ has a $\mathcal{P}$-option which is all $2$s.  That is, the columns containing \colorbox{OI2}{$\star$}s only contain $\mathcal{N}$-positions.

     Suppose $a_1+a_3\geq 3$, and any of $a_1$, $a_2$, or $a_3$ are greater than $5$.  If $a_1+2\cdot a_2 \neq 0\bmod 3$, then we may remove either $1$ or $2$ tokens from piles of size $1$ or $2$ to obtain $a_1+2\cdot a_2 = 0\bmod 3$, which, by induction, is a $\mathcal{P}$-option.
\end{proof}

Based on Lemma \ref{lemma - Version C for small piles}, we think that, if there are sufficient piles of each size, Version C mimics the one-dimensional game \textsc{Subtract}$(1,2)$.  We have verified the following for $a_i\leq 5$ and $n\leq 5$.

\begin{conj}
    Suppose we have a position of Version C where there are $a_1$, $a_2$, \ldots $a_n$ piles of size $1$, $2$, \ldots, $n$, respectively. Suppose further that $a_i\geq 2$ for all $1\leq i\leq n$.  Suppose further that the position is not precisely $1,1,2,2,3,3,3$ (i.e., $(2,2,3)$ in the above lemma).  Then it is a $\mathcal{P}$-position if and only if the total number of tokens is a multiple of $3$.    
\end{conj}

Our last lemma shows a kind of counterexample to this conjecture if we allow large gaps in the sequence $a_1,\ldots, a_n$, in particular, if $a_2=a_3=\ldots =a_{n-1}=0$.

\begin{lem}
    Suppose we have a position with $n$ piles of size $1$ and a single pile of size $n$.  Then it is a $\mathcal{P}$-position precisely when
    \begin{itemize}
        \item $k\geq 2n$ and $k+n = 0 \bmod 3$, or
        \item $k<2n$ and $k+2n = 0 \bmod 4$.
    \end{itemize}
\end{lem}

\begin{figure}[ht!]   \centering

$\begin{array}{c|ccccccccccccc}
\hbox{$n\backslash k$}
&0&1&2&3&4&5&6&7&8&9&10&11&12
\\
\hline
0&\fbox{$\mathcal{P}$}&&&\mathcal{P}&&&\mathcal{P}&&&\mathcal{P}&&&\mathcal{P}
\\
1&&&\fbox{$\mathcal{P}$}&&&\mathcal{P}&&&\mathcal{P}&&&\mathcal{P}&
\\
2&\mathcal{P}&&&&\fbox{$\mathcal{P}$}&&&\mathcal{P}&&&\mathcal{P}&&
\\
3&&&\mathcal{P}&&&&\fbox{$\mathcal{P}$}&&&\mathcal{P}&&&\mathcal{P}
\\
4&\mathcal{P}&&&&\mathcal{P}&&&&\fbox{$\mathcal{P}$}&&&\mathcal{P}&
\\
5&&&\mathcal{P}&&&&\mathcal{P}&&&&\fbox{$\mathcal{P}$}&&
\\
6&\mathcal{P}&&&&\mathcal{P}&&&&\mathcal{P}&&&&\fbox{$\mathcal{P}$}
\end{array}$
    \caption{Positions of Version C with $k$ piles of size $1$ and a single pile of size $n$. 
    The case $k=2n$ is \fbox{boxed}.
}
    \label{figure: Version C for ones and one big pile}
\end{figure}

Note that $k=2n$ implies that both 
$k+n = 0 \bmod 3$ and $k+2n = 0 \bmod 4$.

\begin{proof}  We refer the reader to Figure \ref{figure: Version C for ones and one big pile}, and leave the details as an exercise.
\end{proof}  
From any position on Figure \ref{figure: Version C for ones and one big pile}, the legal moves are up $1$, left $1$, left $2$, up-and-left $1$.  In other words, it is isomorphic to the two-dimensional vector subtraction game defined by $\{(0,-1),(-1,0),(-2,0),(-1,-1)\}$.

\section{Applications to impartial game sums}

Suppose we have finitely many impartial game positions.  If we consider their usual disjunctive sum, where one is allowed to select one of the games and play in it, we obtain the usual theory of $\mathcal{N}$ and $\mathcal{P}$ positions via the \textsc{nim}-sum.

On the other hand, suppose we have two game positions, $G$ and $H$, and we are allowed to make a move, either in $G$, in $H$, or in both.  In this case, we have a $\mathcal{P}$-position if and only if $G$ and $H$ are both $\mathcal{P}$-positions.  In other words, this is isomorphic to playing \textsc{oooooob} if we replace $\mathcal{P}$ with `even' and $\mathcal{N}$ with `odd.'  If we extend this to an arbitrary set of game positions, then this isomorphism still holds for Version A of \textsc{Generalized oooooob}.

\begin{question}
    Is there some game sum operation which gives a similar analogy for Version B or Version C?
\end{question}

Not for usual sums.  For example, $\langle o,o,o\rangle$ is a $\mathcal{P}$-position for Version C, but if we `play Version C' on the \textsc{nim}-position $3,1,1$, we have the $\mathcal{P}$-option $1,1,1$.  Maybe for another kind of sum?


\end{document}